\documentclass{article}
\usepackage{amssymb}
\usepackage{amsfonts}
\usepackage{amsmath}
\usepackage{graphicx}

\newtheorem{theorem}{Theorem}[section]
\newtheorem{corollary}[theorem]{Corollary}
\newtheorem{definition}[theorem]{Definition}

\newtheorem{lemma}[theorem]{Lemma}

\newenvironment{proof}[1][Proof]{\noindent\textbf{#1.} }{\ \hfill \rule{0.5em}{0.5em}}
\newenvironment{oldtext}[1][]{\noindent{\bf (Old text\textellipsis }}{\ \hfill  \bf{\textellipsis )}}
\newenvironment{cmtpar}[1][] {\noindent \bfseries ((}{\hfill ))\normalfont}
\newcommand{\nc}{\newcommand}
\nc{\bd}{\mathrm{bd}}
\nc{\intr}{\mathrm{int}}
\nc{\cl}{\mathrm{cl}}
\nc{\rmcite}[1]{{\rm\cite{#1}}}
\nc{\rmb}[1]{{\rm \bf #1}}
\nc{\spset}{\stackrel{\circ}{\supset}}
\nc{\sbset}{\stackrel{\circ}{\subset}}
\nc{\BOT}{\begin{oldtext}}
\nc{\EOT}{\end{oldtext}}
\nc{\BCP}{\begin{cmtpar}}
\nc{\ECP}{\end{cmtpar}}

\begin{document}

\title{Sets of tetrahedra, \\
defined by maxima of distance functions}
\author{Jo\"{e}l Rouyer \and Costin V\^{\i}lcu}
\maketitle

\begin{abstract}
We study tetrahedra and the space of tetrahedra from the viewpoint of local
and global maxima for intrinsic distance functions. \newline
\textsc{C}45. \newline
\end{abstract}


\section{Introduction and main result}

One direction in present differential geometry is the study of
differentiable spaces with singularities, of which polyhedral convex
surfaces are simple examples. Surprisingly, even the geometry of tetrahedra
is not completely known, see for example the recent results of V. A.
Zalgaller \cite{Zalgaller}. On the other hand, the geometry of convex
polyhedra is an important part of computational geometry, see for example
the book by E. D. Demaine and J. O'Rourke \cite{Demaine_O'Rourke}.

\medskip

Consider a tetrahedron $T$ in the Euclidean space $\mathbb{R}^{3}$, endowed
with the intrinsic distance $\rho$ induced by the distance in $\mathbb{R}%
^{3} $. The metric $\rho$ is defined, for any points $x,y$ in $T$, as the
length of a \emph{segment} (\textit{i.e.}, shortest path on $T$) from $x$ to 
$y$. For a given point $x\in T$, let $\rho_{x}$ be the \textit{distance
function from} $x$, defined by $\rho_{x}(y)=\rho(x,y)$. When necessary, we
shall write $\rho^S$ to emphasize the surface $S$.

Denote by $M_{x}$ the set of local maxima of $\rho_{x}$, and by $F_{x}$ the
set of global maxima of $\rho_{x}$ (or \emph{farthest points} from $x$, or 
\emph{antipodes} of $x$).

The study of farthest points and, more generally, of local maxima for
distance functions on convex surfaces, has its origin in several questions
of H. Steinhaus, see $\S $A35 in \cite{cfg}. The paper \cite{v2} surveys
fundamental results on this subject; more recent work is due to T.
Zamfirescu, the authors and their collaborators.

\medskip

Define the following subsets of the space $\mathcal{T}$, of all the
tetrahedra in $\mathbb{R}^{3}$ up to isometry and homothety: 
\begin{equation*}
\mathcal{M}_{n}=\left\{ T\in\mathcal{T}|\exists x\in T~\#M_{x}\geq n\right\},
\end{equation*}
\begin{equation*}
\mathcal{F}_{n}=\left\{ T\in\mathcal{T}|\exists x\in T~\#F_{x}\geq n\right\},
\end{equation*}
where $\#S$ denotes the cardinality of the set $S$.

We have $\mathcal{M}_{k} \supset\mathcal{M}_{k+1} \cup\mathcal{F}_{k}$,
directly from the definitions, wherefrom the following diagram. 
\begin{equation*}
\begin{array}{rrcccccccccrr}
{\mathcal{T}} & {\supset} & {\mathcal{M}_{2}} & {\supset} & {\mathcal{M}_{3}}
& {\supset} & {\mathcal{M}_{4}} & {\supset} & {\mathcal{M}_{5}} & {\supset}
& {\mathcal{M}_{6}} & {\supset} & {...} \\ 
{} & {} & {\cup} & {} & {\cup} & {} & {\cup} & {} & {\cup} & {} & {\cup} & {}
& {} \\ 
{} & {} & {\mathcal{F}_{2}} & {\supset} & {\mathcal{F}_{3}} & {\supset} & {%
\mathcal{F}_{4}} & {\supset} & {\mathcal{F}_{5}} & {\supset} & {...} &  & 
\end{array}%
\end{equation*}

The equality $\mathcal{M}_{2}=\mathcal{F}_{2}$ holds for polyhedral convex
surfaces, see \cite{vz2}. Moreover, for polyhedral convex surfaces with $n$
vertices, we have $\mathcal{F}_{n}\neq\varnothing$ \cite{MonMien5} and $%
\mathcal{F}_{n+1}=\varnothing$ \cite{RS1}. In \cite{MonMien5} it is also
proven that $\mathcal{M}_{n+2}= \varnothing$, though this is not emphasized
as a theorem. For tetrahedra, it has been proven that $\mathcal{F}_{2}=%
\mathcal{T}$ \cite{MonMien4}.

We examine each inclusion in the above diagram and make precise the
\textquotedblleft re\-la\-tive size\textquotedblright of the subset. We
write $A\spset B$ whenever $A\supset B$ and both $B$ and $A\setminus B$ have
interior points. If $B$ is nowhere dense in $A$ (\textit{i.e.}, $\mathrm{int}%
\left( \mathrm{cl}\left(B\right) \right) =\varnothing$) we write $A\overset{%
n.d.}{\supset}B$.

\bigskip

\noindent\textbf{Theorem } \textsl{All sets $\mathcal{M}_{k}$ are open in $%
\mathcal{T}$ and we have the following diagram} 
\begin{equation*}
\begin{array}{rrcccccccccrr}
{\mathcal{T}} & {=} & {\mathcal{M}_{2}} & {\supsetneq} & {\mathcal{M}_{3}} & 
{\spset} & {\mathcal{M}_{4}} & {\spset} & {\mathcal{M}_{5}} & \overset{n.d.}{%
{\supset}} & {\mathcal{M}_{6}} & {=} & {\varnothing} \\ 
&  & {\parallel} &  & {\cup\circ} &  & {\cup}^{{n.d}} &  & {\cup\circ} &  & 
&  &  \\ 
&  & {\mathcal{F}_{2}} & {\spset} & {\mathcal{F}_{3}} & \overset{n.d.}{{%
\supset}} & {\mathcal{F}_{4}} & \overset{n.d.}{{\supset}} & {\mathcal{F}_{5}}
& = & {\varnothing.} &  & 
\end{array}%
\end{equation*}

\bigskip

We do not know if the set $\mathcal{M}_{2} \setminus\mathcal{M}_{3}$ has
interior points in $\mathcal{T}$.

Parts of this result hold as well for general polyhedral convex surfaces. We
choose to give them for tetrahedra, for the sake of a unified presentation.

The proof of our theorem is split into several lemmas, treating each
inclusion separately. In Section \ref{Sect_lmax} we study the chain of
inclusions concerning the sets $\mathcal{M}_{k}$, in Section \ref{Sect_glmax}
we treat the inclusions $\mathcal{F}_{k}\subset\mathcal{M}_{k}$, while in
Section \ref{Sect_gmax} we consider the chain of inclusions concerning the
sets $\mathcal{F}_{k}$.

Our framework is the space of tetrahedra, and our main tools are properties
of cut loci, gluings and unfoldings; they are all presented in Section \ref%
{Preliminaries}.

\bigskip

For the reader's convenience we recall below a few definitions and give
additional notation. A \emph{very acute vertex} of a polyhedral surface is a
vertex $v$ the total angle $\theta_{v}$ of which is less than $\pi$. The 
\emph{curvature} $\omega_{v}$ at the vertex $v$ of a polyhedral surface is
defined as $\omega_{v} = 2\pi - \theta_v$. A tetrahedron is called \emph{%
isosceles} if any opposite edges are equal or, equivalently, if the total
angle at every vertex is equal to $\pi$. The closure, interior, and boundary
of any subset $N$ of a topological space are respectively denoted by $%
\mathrm{cl}(N)$, $\mathrm{int}(N)$, and $\mathrm{bd}(N)$. In any metric
space, $B(x,r)$ stands for the closed ball centered at $x$ of radius $r>0$,
and $S(x,r)$ for its boundary. As usual, $||\cdot||$ denotes the standard
norm in $\mathbb{R}^{3}$.


\section{Preliminaries}

\label{Preliminaries}

In this section we briefly present some basic tools employed in the
intrinsic geometry of polyhedra. We will use them most often implicitly.


\subsection{Space of tetrahedra}

By a \textit{tetrahedron}, we mean the boundary of the convex hull of four
non-coplanar points. Let $\mathfrak{T}$ denote the space of all tetrahedra.
We endow it, as usual in convex geometry, with the Pompeiu-Hausdorff
distance $d_{PH}$. Since a tetrahedron is determined by its vertex set, $%
\mathfrak{T}$ is a quotient of an open set of $\mathbb{R}^{12}$ by the free
and properly discontinuous action of the group $S_{4}$ of permutations.

By $\mathcal{T}$ we denote the space of all tetrahedra, up to isometry and
homothety; it is endowed with the induced topology from $\mathfrak{T}$.

\begin{lemma}
\label{dim_T} There exists an open and dense set $\mathcal{O}$ in $\mathcal{T%
}$, which carries the structure of a $5$-dimensional differentiable manifold.
\end{lemma}

\begin{proof}
Notice that a tetrahedron $T$ in $\mathfrak{T}$ which is invariant under an
isometry of $\mathbb{R}^3$ has two congruent faces.

Denote by $\mathfrak{I}$ the subset of $\mathfrak{T}$ consisting of all
tetrahedra with two congruent faces, and put $\mathfrak{O}=\mathfrak{T}%
\setminus \mathfrak{I}$. Clearly, $\mathfrak{I}$ is closed and $\mathfrak{O}$
is open in $\mathfrak{T}$. Moreover, each tetrahedron $T\in \mathfrak{I}$
can easily be approximated with tetrahedra $T_{n}$, each of which has
distinct edge lengths; therefore, $\mathfrak{O}$ is dense in $\mathfrak{T}$.

The group $G$ generated by affine homotheties and affine isometries of $%
\mathbb{R}^{3}$ acts on $\mathfrak{O}$ in a natural way. The action is
clearly smooth, proper, and free by the definition of $\mathfrak{O}$. Hence $%
\mathcal{O}=\mathfrak{O}/G$ is a smooth $5$-dimensional manifold (see, 
\textit{e.g.}, Theorem 5.119 in \cite{JL}). The density of $\mathcal{O}$ in $%
\mathcal{T}$ follows from the density of $\mathfrak{O}$ in $\mathfrak{T}$.
\end{proof}

\bigskip

With some abuse of notation, we say that the dimension of $\mathcal{T}$ is $%
5 $.

\bigskip

The completions $\overline{\mathfrak{T}}$ of $\mathfrak{T}$ and $\overline {%
\mathcal{T}}$ of $\mathcal{T}$ include, beside tetrahedra, doubles of planar
convex quadrilaterals or triangles. Notice that a doubly covered triangle
appears in these complete spaces infinitely many times, once for each
position of its fourth, flat vertex.

We introduce $\overline{\mathfrak{T}}$ and $\overline{\mathcal{T}}$ for
practical reasons: we aim to construct suitable doubly covered
quadrilaterals or triangles, and afterwards to perturb their vertices to get
non-degenerate tetrahedra. Clearly, the convergence in $\overline{\mathfrak{T%
}}$ implies convergence in $\overline{\mathcal{T}}$.

\medskip

Define the following subsets of the space $\overline{\mathcal{T}}$: 
\begin{equation*}
\tilde{\mathcal{M}}_{n}=\left\{ T \in\overline{\mathcal{T}} |\exists x\in
T~\#M_{x}\geq n\right\} ,
\end{equation*}
\begin{equation*}
\tilde{\mathcal{F}}_{n}=\left\{ T \in\overline{\mathcal{T}} |\exists x\in
T~\#F_{x}\geq n\right\} ,
\end{equation*}
and notice that $\tilde{\mathcal{M}}_{k} \supset\tilde{\mathcal{M}}_{k+1}
\cup\tilde{\mathcal{F}}_{k}$ holds in $\overline{\mathcal{T}}$.

\medskip

The next classical lemma will implicitly be used for the convergence of
tetrahedra.

\begin{lemma}
\rmcite{{BCS}} \label{LIDCv} Let $S_{n}$ be a sequence of convex
surfaces, converging to $S$ with respect to the Pompeiu-Hausdorff metric.
Let $x_{n},y_{n}$ be points in $S_{n}$, converging to $x,y\in S$, and $%
\gamma_{n}$ segments on $S_{n}$ joining $x_{n}$ to $y_{n}$, converging to $%
\gamma\subset S$. Then $\gamma$ is a segment on $S$ joining $x$ to $y$, and $%
\rho^S\left( x,y\right) =\lim_{n}\rho^{S_n}\left( x_{n},y_{n}\right)$.
\end{lemma}


\subsection{Gluing and unfolding}

\label{Unfold}

One of our main tools is the Alexandrov's gluing theorem, given next (see 
\cite{code2}, p.100).

\begin{lemma}
\label{Alexandrov} \rmb{{(Alexandrov's gluing theorem)}} Consider
a topological sphere $S$ obtained by gluing planar polygons (i.e., naturally
identifying pairs of sides of the same length) such that at most $2\pi$
angle is glued at each point. Then $S$, endowed with the intrinsic metric
induced by the distance in $\mathbb{R}^{2}$, is isometric to a polyhedral
convex surface $P\subset\mathbb{R}^{3}$, possibly degenerated. Moreover, $P$
is unique up to rigid motion and reflection in $\mathbb{R}^{3}$.
\end{lemma}

The \textit{double} $D$ of the planar convex body $K$ is obtained by gluing
two isome\-tric copies of $K$ along their boundary, by identifying the
naturally corresponding points. The two copies of $K$ are called the \textit{%
faces} of $D$. With some abuse of notation, we shall identify $K$ with one
face of $D$. $D$ can be seen as limit of convex surfaces in $\mathbb{R}^{3}$.

Opposite to Alexandrov's gluing theorem are the results on unfoldings. For
this subject, see the excellent book \cite{Demaine_O'Rourke} by E. Demaine
and J. O'Rourke.

\begin{definition}
Let $P$ be a convex polyhedron. The \emph{cut locus} $C(x)$ of the point $x$
on the polyhedral convex surface $P$ is defined as the set of endpoints,
different to $x$, of all nonextendable shortest paths (on the surface $P$)
starting at $x$.
\end{definition}

We postpone the presentation of the properties of cut loci to the next
subsection.

\begin{lemma}
\label{Source_unf} \rmcite{{ss-spps-86}} \rmb{{(Source
unfolding)}} Consider a polyhedral convex surface $P$ and a point $x$ in $P$%
. Cutting along $C(x)$ produces a surface $U$ with boundary which can be
unfolded in the plane without overlappings. Moreover, $U$ is star-shaped
with respect to $x$.
\end{lemma}

\begin{lemma}
\label{Star_unf} \rmcite{{unfold1}} \rmb{{(Star unfolding)}}
Consider a polyhedral convex surface $P$ and a point $x$ in $P$ which is
joined to each vertex of $P$ by a unique segment. Cutting along the union of
those segments produces a surface $U$ with boundary which can be unfolded in
the plane without overlappings. Assume, with some abuse of notation, that $%
U\subset\mathbb{R}^{2}$; then the image of $C(x)$ in $U$ after unfolding is
precisely the restriction to $U$ of the Voronoi diagram of the images of $x$.
\end{lemma}

Roughly speaking, the following lemma shows that close unfoldings must fold
to close tetrahedra.

\begin{lemma}
\label{Close_unf} Let $\left\{ P^{n}\right\} _{n\in\mathbb{N}}=\left\{
a_{1}^{n}a_{2}^{n}\ldots a_{p}^{n}\right\} _{n\in\mathbb{N}}$ be a sequence
of $p$-gons in $\mathbb{R}^{2}$ converging to the $p$-gon $P=a_{1}\ldots
a_{p}$, such that $a_{j}^{n}\rightarrow a_{j}$ for $j=1,...,p$. Endow each $%
P_{n}$ and $P$ with compatible sets $R_{n}$ and $R$ of gluing rules, meaning
that $a_{i}^{n}a_{j}^{n}$ is glued (i.e., identified preserving the length)
to $a_{k}^{n}a_{l}^{n}$ according to $R_{n}$ if and only if $a_{i}a_{j}$ is
glued to $a_{k}a_{l}$ according to $R$. Assume that, for each $n$, the
gluing of $P_{n}$ according to $R_{n}$ yields a tetrahedron $T_{n}$. Then
the gluing of $P$ according to $R$ yields a tetrahedron $T$ (possibly
degenerated), and $T_{n}\rightarrow T$.
\end{lemma}

\begin{proof}
By hypothesis, the gluing of $P_{n}$ according to $R_{n}$ yields a
polyhedron $T_{n}$ with only four vertices of positive curvature, say $%
v_{1}^n,v_{2}^n,v_{3}^n,v_{4}^n$, corresponding to the points $%
a_{i_{1}}^{n},a_{i_{2}}^{n},a_{i_{3}}^{n},a_{i_{4}}^{n}$ respectively.

By continuity, the limits of the zero curvature vertices of $T_{n}$ also
have zero curvature, whence $T$ has at most $4$ vertices, $v_i =\lim_n v_i^n$%
, $i=1,...,4$.

Notice now that the distance in $T_{n}$ between the vertices $v_{i}^{n}$, $%
v_{j}^{n}$, depends continuously on the position of the vertices $\left(
a_{1}^{n},\ldots,a_{p}^{n}\right)$. In other words, there exist continuous
maps $f_{ij}:\mathbb{R}^{2p}\rightarrow\mathbb{R}$, depending only of the
gluing rules, such that $\rho^T \left( v_{i},v_{j}\right)
=f_{ij}\left(a_{1},\ldots,a_{p}\right)$ and 
$\rho^{T_n} \left(v_{i}^{n},v_{j}^{n}\right)=f_{ij}\left( a_{1}^{n},\ldots,a_{p}^{n}\right)$, $%
n\in\mathbb{\mathbb{N}}$. Hence, the distance on $T_{n}$ between two
vertices converges to the distance on $T$ between the limit vertices. Since
the relative positions of a tetrahedron's vertices depend continuously on
its edge lengths, $T_{n}$ converges to $T$ in $\mathcal{T}$ .
\end{proof}


\subsection{Cut loci and local maxima}

\label{cl}

A very important tool for the study of distance functions on convex
polyhedra is the cut locus. See \cite{sa} for some of its applications in
Riemannian geometry.

We have already presented the definition of cut loci in the previous
subsection. Basic properties of cut loci are given in the following lemma.

\begin{lemma}
\label{basic} Let $x$ be a point on a polyhedral convex surface $P$.

(i) $C(x)$ is a tree whose leaves (endpoints) are vertices of $P$, and all
vertices of $P$, excepting $x$ (if the case), are included in $C(x)$.

(ii) The junction points in $C(x)$ are joined to $x$ by as many segments as
their degree in the tree, each leaf of $C(x)$ is joined to $x$ by one
segment, and the other points in $C(x)$ are joined to $x$ by precisely two
segments.

(iii) The edges of $C(x)$ are shortest paths on $P$.

(iv) Assume the shortest paths $\Gamma$ and $\Gamma^{\prime}$ from $x$ to $%
y\in C(x)$ are bounding a domain $D$ of $P$, which intersects no other
shortest path from $x$ to $y$; the case $\Gamma= \Gamma^{\prime}$ is also
possible. Then the arc of $C(x)$ at $y$ towards $D$ bisects the angle of $D$
at $y$.
\end{lemma}

The properties (i)-(ii) and (iv) are well known in a more general framework,
while (iii) is Lemma 2.4 in \cite{code1}.

\bigskip

Lemma \ref{basic} yields, for the particular case of polyhedra with $4$
vertices, the following lemma.

\begin{lemma}
\label{L2} The cut locus of a point $x$ in a (possibly degenerated)
tetrahedron $T$ is homeomorphic to one of the letters $H$, $X$, $Y$, or $I$.
\end{lemma}

The case of a Jordan arc cut locus is very particular, as stated by the next
known result (see \cite{INV} for a proof).

\begin{lemma}
\label{deg} If the cut locus $C(x)$ of a point $x$ in a convex (possibly
degenerated) polyhedron $P$ is an arc then $P$ is a doubly covered polygon.
\end{lemma}

We will give in the following some basic properties of local maxima.

\begin{lemma}
\label{loc-max} Let $P$ be a polyhedral convex surface and $x, y$ points in $%
P$.

(i) $M_{x} \subset C \left( x\right) $.

(ii) $y\in M_{x}$ if and only if the angle at $y$ between any two
consecutive segments from $y$ to $x$ is less than $\pi$. Consequently, $%
M_{x} $ contains only strict local maxima.

(iii) Each point in $M_{x}$ is either a leaf or a junction point in $C\left(
x\right) $ (a vertex of degree two in $C\left( x\right)$ is also considered
as a junction point).

(iv) A leave $y$ of $C \left( x\right) $ belongs to $M_{x}$ if and only if $%
\omega_{y} > \pi$.
\end{lemma}

\begin{proof}
The inclusion $M_{x} \subset C \left( x\right) $ is clear from the
definitions.

The second statement is straightforward, if one considers a local unfolding
of $P$ along the segments from $x$ to $y$.

From (ii) in Lemma \ref{basic} and (ii) we obtain (iii) and (iv).
\end{proof}


\section{The inclusions $\mathcal{M}_{k} \supset\mathcal{M}_{k+1}$}

\label{Sect_lmax}

In this section we treat the first chain of inclusions in our Theorem,
concerning the sets $\mathcal{M}_{k}$. To each inclusion one or several
lemmas are devoted.

\begin{lemma}
\label{LMbB} Let $X$ be a compact topological space, and $f:X\rightarrow 
\mathbb{R}$ a continuous function. Let $U$ be an open set in $X$, and $x\in
X $ a point such that $f\left( x\right) >\max_{y\in\bd (U)}f\left( y\right) $%
. Then $f$ admits a local maximum in $U$.
\end{lemma}

\begin{proof}
Let $z$ be a global maximum of $f$ on $\cl(U)$. Clearly, $z\notin\bd(U)$ for 
$f\left( z\right) <f\left( x\right) $, hence it is a global maximum for the
restriction of $f$ to $U$ and, consequently, a local maximum for $f$.
\end{proof}

\begin{lemma}
\label{perturb_max} Let $S_n$ a sequence of convex surfaces, converging to $%
S $ with respect to the Pompeiu-Hausdorff metric. Let $x_{n},y_{n}$ be two
points of $S_n$ converging to $x\in S$ and $y\in S$ respectively. Let $y$ be
a strict local maximum of $\rho^S_x$. Then for each $r>0$, there exists an
index $N$ such that for all $n>N$, the function $\rho^{S_n}_{x_n}$ admits a
local maximum in $B\left( y_n,r\right) $.
\end{lemma}

\begin{proof}
Since $y$ is a local maximum for $\rho^S_{x}$, we can assume (by making $r$
smaller) that 
\begin{equation*}
\varepsilon\overset{def}{=}\rho^S_{x}\left( y\right) -\max_{z\in S\left(
y,r\right) }\rho^S_{x}\left( z\right) >0\text{.}
\end{equation*}
Let $z_{n}$ be a global maximum for $\rho^{S_n}_{x_{n}}$ restricted to $%
S\left( y_{n},r\right) $. For $n$ large enough, we have 
\begin{align}
\left\vert \rho^{S_n}\left( x_{n},y_{n}\right) -\rho^S\left( x,y\right)
\right\vert & <\varepsilon/2  \label{1} \\
\rho^{S_n}\left( x_{n},z_{n}\right) & \leq\lim\sup\rho^{S_n}\left(
x_{n},z_{n}\right) +\varepsilon/2=\rho^S\left( x,z^{\prime}\right)
+\varepsilon /2  \notag
\end{align}
for some limit point $z^{\prime}$ of $\left\{ z_{n}\right\}$. By Lemma \ref%
{LIDCv}, such a point $z^{\prime}$ must belong to $S\left( x,r\right) $,
whence 
\begin{equation*}
\rho^{S_n}\left( x_{n},z_{n}\right) \leq\rho^S_{x}\left( y\right)
-\varepsilon /2\text{,}
\end{equation*}
which, together with (\ref{1}), leads to 
\begin{equation*}
\rho^{S_n}_{x_{n}}\left( z\right) \leq\rho^{S_n}_{x_{n}}\left( z_{n}\right)
<\rho^{S_n}_{x_{n}}\left( y_{n}\right)
\end{equation*}
for any point $z\in S\left( y_{n},r\right) $. Hence, by Lemma \ref{LMbB}, $%
\rho^{S_n}_{x_{n}}$ has a local maximum in $B\left( y_{n},r\right) $.
\end{proof}

\bigskip

Lemmas \ref{perturb_max} and \ref{loc-max} (ii) imply the following

\begin{corollary}
\label{Open_M} For any polyhedral convex surface, each set $\mathcal{M}_{k}$
is open in $\mathcal{T}$ and each set $\tilde{\mathcal{M}}_{k}$ is open in $%
\overline{\mathcal{T}}$, $k \geq2$.
\end{corollary}

\begin{lemma}
\label{Lm3} The isosceles tetrahedra belong to $\mathcal{M}_{2} \backslash 
\mathcal{M}_{3}$.
\end{lemma}

\begin{proof}
Let $T$ be an isosceles tetrahedron and $x$ a point in $T$.

Notice that a vertex $v$ of $T$ which is also a leaf of $C(x)$ cannot be a
local maximum for $\rho_{x}$. Indeed, the total angle of $T$ at $v$ is $\pi$%
, and there is only one segment from $x$ to $v$. Therefore (see Lemma \ref%
{L2}), if $C(x)$ is an $H$-tree then $\#M_{x} \leq2$, if $C(x)$ is an $X$%
-tree then $\#M_{x} = 1$, and if $C(x)$ is an $Y$-tree then $\#M_{x} \leq2$.
\end{proof}

\begin{lemma}
$\mathrm{int} \left( \mathcal{M}_{3} \setminus\mathcal{M}_{4} \right)
\neq\varnothing$.
\end{lemma}

\begin{proof}
Recall that all cut loci on tetrahedra have at most two flat junction
points, by Lemma \ref{L2}. This and Lemma \ref{loc-max} imply that each
tetrahedron in $\mathcal{M}_{4}$ has at least two very acute vertices.
Therefore, it suffices to provide an open subset of $\mathcal{M}_{3}$, each
tetrahedron of which has at most one very acute vertex.

Let $R$ be a non-square rectangle of center $o$. The cut-locus $%
C\left(o\right)$ on the double of $R$ is a $H$-tree and its junctions points 
$u,w$ clearly belong to $M_{o}$. Consider a quadrilateral $Q$ close to $R$,
with one vertex acute and other three obtuse, and a point $x$ on the double $%
D$ of $Q$, close to $o$. By Lemma \ref{perturb_max}, there is at least one
point of $M_{x}$ near $u$, and another one near $w$. Of course, the acute
vertex of $Q$ also belongs to $M_{x}$, whence $D\in\tilde{\mathcal{M}_{3}}$.
Now, tetrahedra close enough to $D$ cannot belong to $\mathcal{M}_{4}$
because they have only one very acute vertex, and so they are interior to $%
\mathcal{M}_{3}\backslash\mathcal{M}_{4}$, by Corollary \ref{Open_M}.
\end{proof}

\begin{lemma}
$\mathrm{int} \left( \mathcal{M}_{4} \setminus\mathcal{M}_{5}\right)
\neq\varnothing$.
\end{lemma}

\begin{proof}
Lemmas \ref{L2} and \ref{loc-max} imply that each tetrahedron in $\mathcal{M}%
_{5}$ has at least three very acute vertices. Therefore, it suffices to
provide an open subset of $\mathcal{M}_{4}$, each tetrahedron of which has
at most two very acute vertices.

Consider a planar convex quadrilateral $Q=v_{1}v_{2}v_{3}v_{4}$ such that $%
v_{1},v_{2},v_{3},v_{4}$ are consecutive vertices in a regular hexagon of
centre $x$.

Then $M_{x} = \{v_{1}, v_{2}, v_{3}, v_{4}\}$ on the double $R$ of $Q$,
hence $R \in\tilde{\mathcal{M}}_{4}$. Since $R$ has only two very acute
vertices, we get $R \not \in \tilde{\mathcal{M}}_{5}$.

To end the proof, notice that small perturbations of $R$ in $\overline {%
\mathcal{T}}$ provide tetrahedra with precisely two very acute vertices,
which are therefore in $\mathcal{M}_{4} \setminus \mathcal{M}_5$ by
Corollary \ref{Open_M}.
\end{proof}

\begin{lemma}
$\mathrm{int} \left( \mathcal{M}_{5}\right) \neq\varnothing$.
\end{lemma}

\begin{proof}
By constructing an explicit example, we show next that $\tilde{\mathcal{M}}%
_{5} \neq\varnothing$. This and Corollary \ref{Open_M} would then imply $%
\mathrm{int} \left( \mathcal{M}_{5}\right) \neq\varnothing$.

Consider the planar polygon $L$ drawn in Figure \ref{F1}, where the
line-segments marked with X are all equal, as are those marked with II. The
two big rhombi have angles $\pi/4$ and $3\pi/4$, and the smaller one to the
right has angles $\pi/8$ and $7\pi/8$. Our figure has horizontal symmetry. 
\begin{figure}[th]
\centering
\includegraphics[width=0.65\textwidth]{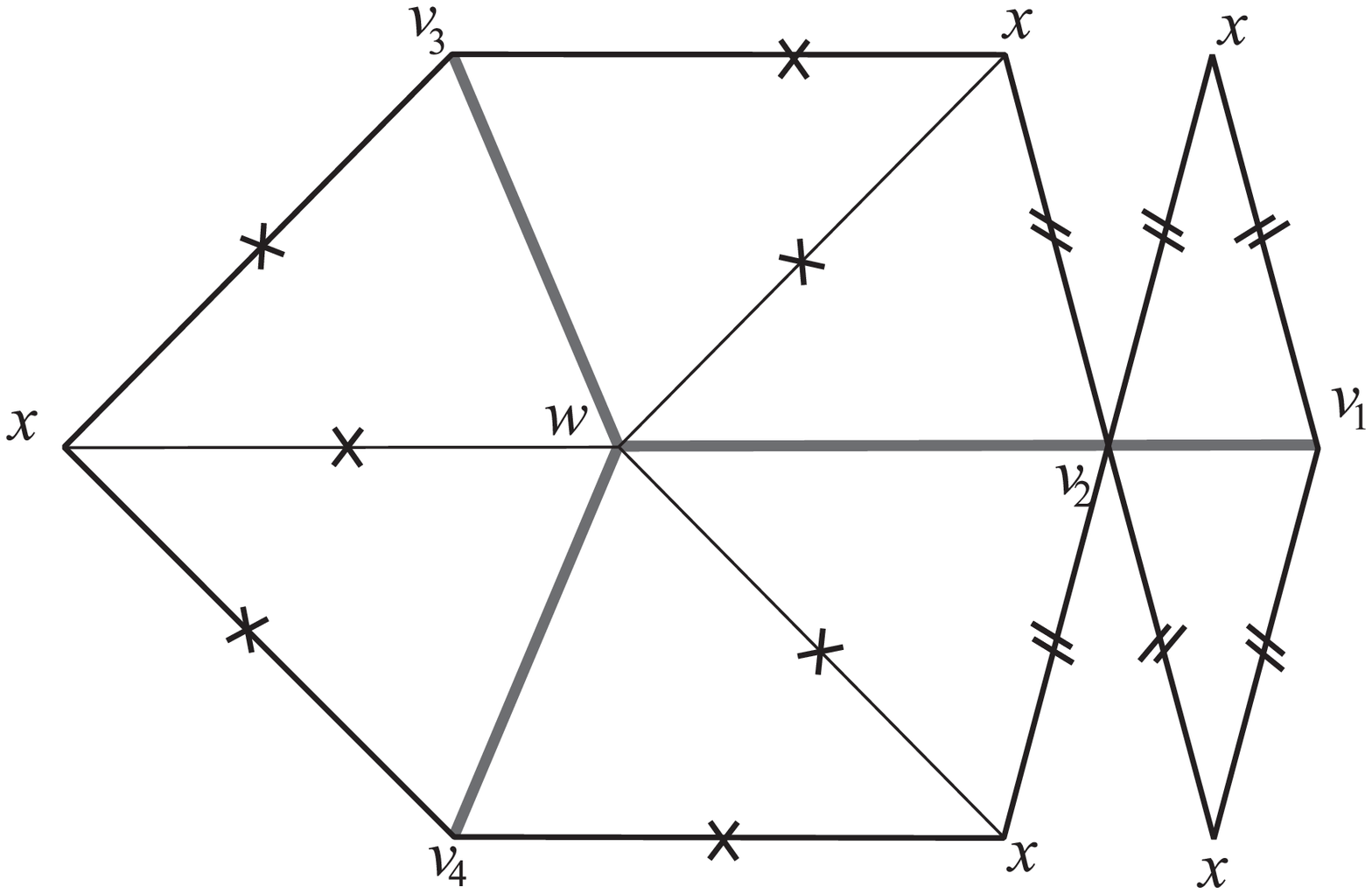}
\caption{Unfolding of a tetrahedron $T$ in $\tilde{\mathcal{M}}_{5}$,
obtained by cutting along the segments from $x\in T$ to the vertices of $T$. 
}
\label{F1}
\end{figure}

In order to obtain a polyhedral convex surface $T$, glue the sides of $L$ as
follows: identify the two sides incident to $v_{i}$, $i=1,3,4$, and identify
the upper sides incident to $v_{2}$, as well as the lower sides incident to $%
v_{2}$. These identifications are all possible, due to the length equalities.

Notice that the total angle $\theta_{i}$ at the point $v_{i}$ after gluing
verifies $\theta_{1}=7\pi/8$, $\theta_{2}=13\pi/8$, and $\theta_{3}=%
\theta_{4}=3\pi/4$. Moreover, the total angle $\theta_{x}$ after gluing at $%
x $ verifies $\theta_{x}=2\pi$. By Alexandrov's gluing theorem, the
resulting surface $T$ is a tetrahedron, and the point in $T$ corresponding
to $x\in L$, also denoted by $x$, is a flat point.

One can easily notice that the segments (on $T$) joining $x$ to $w$, $v_{1}$%
, $v_{2}$, $v_{3}$, respectively $v_{4}$, correspond to marked line-segments
in Figure \ref{F1}. Therefore, on $T$, $\angle xv_{2}w<\pi/2$, $\angle
xv_{2}v_{1}<\pi/2$; these inequalities and $\theta_{j}<\pi$, $j=1,3,4$,
imply $M_{x}=\left\{ w,v_{1},v_{2},v_{3},v_{4}\right\}$ (by Lemma \ref%
{loc-max}), whence $T\in\tilde{\mathcal{M}}_{5}$ and the conclusion follows.
\end{proof}

\bigskip

The next lemma follows from Lemma \ref{Lm3} and a more general result in 
\cite{ZamTAMS}.

\begin{lemma}
\label{g} Let $S$ be a convex surface and $x$ a point in $S$. Then $M_{x}$
is contained in a minimal (by inclusion) $Y$-subtree of $C(x)$, possibly
degenerated to an arc or a point.
\end{lemma}

\begin{lemma}
$\mathcal{M}_{6}=\varnothing$.
\end{lemma}

\begin{proof}
This follows from Lemmas \ref{loc-max} and \ref{g}.
\end{proof}


\section{The inclusions $\mathcal{M}_{k} \supset\mathcal{F}_{k}$}

\label{Sect_glmax}

\bigskip

In this section we consider the inclusions $\mathcal{F}_{k}\subset \mathcal{M%
}_{k}$, for $k=3,4$.

\begin{lemma}
\label{thintetra} Let $T=abcd\in\overline{\mathfrak{T}}$ be such that $%
\left\Vert c-\frac{a+b}2 \right\Vert \leq \frac12 \left\Vert a-b\right\Vert
\sin\frac{\pi}{16}$ and $\left\Vert d-\frac{a+b}2 \right\Vert \leq \frac12
\left\Vert a-b\right\Vert \sin\frac{\pi}{16}$. Then $\cup_{x\in
T}F_x=\left\{ a,b\right\}$.
\end{lemma}

\begin{proof}
We can assume without loss of generality that $\left\Vert a-b\right\Vert =1$%
. Put $m=\frac{a+b}{2}$ and $\varepsilon=\frac{1}{2}\sin\frac{\pi}{16}$.

We claim that the total angles at $a$ and $b$ are less than $\frac{\pi}{4}$.
By the law of sines, 
\begin{equation*}
\sin\angle cam=\frac{\left\Vert c-m\right\Vert }{\left\Vert a-m\right\Vert }%
\sin\angle acm\leq2\varepsilon\text{.}
\end{equation*}
Similarly $\sin\angle dam \leq 2 \varepsilon$ and, since $\angle
cad\leq\angle dam+\angle cam$, 
\begin{equation*}
\theta_{a}\leq 4 \arcsin\left( 2\varepsilon\right)=\frac{\pi}{4}.
\end{equation*}

By the claim, $\omega_a +\omega_b \geq \frac{7\pi}2$, hence $\omega_c
+\omega_d \leq \frac{\pi}2$ and consequently the total angles at $c$ and $d$
are larger than $\frac{3\pi}2$.

Assume that $y\in F_{x}\backslash\left\{ a,b\right\} $ for some point $x\in
T $.

If $y$ is a vertex, there are at least two segments from $x$ to $y$ (see
Lemma \ref{loc-max}), dividing $T$ into two domains, each of which must
contains at least one vertex. Consider segments $\gamma_{xa}$ and $%
\gamma_{ay}$. These segments separate the domain containing $a$ into two
triangles, at most one of which contains the last vertex. Hence one of these
triangles is a (folded) Euclidean triangle.

If $y$ is a flat point, there are at least three segments from $x$ to $y$,
and $T$ is divided into three domains, each of which contains at least one
vertex. Once again, the domain containing $a$ has at most one interior
vertex distinct to $a$, so there exists a Euclidean triangle with vertices $%
a $,$x$,$y$. Let $\alpha$ be the angle of this triangle at point $a$; the
law of cosines gives%
\begin{equation*}
0\leq\rho\left( x,y\right) ^{2}-\rho\left( x,a\right) ^{2}=\rho\left(
y,a\right) \left( \rho\left( y,a\right) -2\rho\left( x,a\right)
\cos\alpha\right) \text{.}
\end{equation*}

Since $\alpha<\theta_{a}\leq\pi/4$, $\rho\left( y,a\right) \geq\sqrt{2}%
\rho\left( a,x\right) $, and similarly $\rho\left( y,b\right) \geq\sqrt {2}%
\rho\left( b,x\right) $, hence 
\begin{equation*}
\rho\left( y,b\right) +\rho\left( y,a\right) \geq\sqrt{2}\text{.}
\end{equation*}

In remains to prove that no point $y$ can satisfy this inequality.

If $y$ belongs to $abc$ then 
\begin{align*}
\rho\left( y,a\right) +\rho\left( y,b\right) & \leq\rho\left(c,a\right)
+\rho\left( c,b\right) \\
& <\rho\left( a,m\right) +\rho\left( b,m\right) +2 \rho\left(c,m\right) \leq
1+2\varepsilon\text{.}
\end{align*}
The proof is the same for $y \in abd$. If $y$ belongs $acd$, let $e$ be the
intersection of the line $ay$ with the edge $cd$. Obviously, 
\begin{align*}
\rho\left( y,b\right) +\rho\left( y,a\right) & \leq\rho\left( e,b\right)
+\rho\left( e,a\right) \\
& \leq\max\left( \rho\left( c,b\right) ,\rho\left( d,b\right) \right)
+\max\left( \rho\left( c,a\right) ,\rho\left( d,a\right) \right) \\
& <\rho\left( b,m\right) +\rho\left( a,m\right) +2\max\left( \rho\left(
c,m\right) ,\rho\left( d,m\right) \right) \leq 1+2\varepsilon\text{.}
\end{align*}
The proof is similar for $bcd$. Since $1+2\varepsilon<\sqrt{2}$, we get a
contradiction.
\end{proof}

\begin{lemma}
\label{M3-F3} $\mathrm{int}\left( \mathcal{M}_{3}\setminus\mathcal{F}%
_{3}\right) \neq\varnothing$.
\end{lemma}

\begin{proof}
Consider an isosceles triangle $V=v_{1}v_{2}v_{3}$ such that $%
||v_{1}-v_{2}||=||v_{1}-v_{3}||$ and $||v_{1}-\frac{v_{2}+v_{3}}{2}||<\frac{1%
}{2}\sin\frac{\pi}{16}||v_{1}-v_{3}||$, and denote by $D$ its double.

Denote by $u$ the foot of $v_{1}$ on $v_{2}v_{3}$, and notice that $%
M_u=\{v_{1}, v_{2}, v_{3}\}$. Therefore, $D\in\tilde{\mathcal{M}}_{3}=%
\mathrm{int}\left( \tilde{\mathcal{M}}_{3}\right)$ (see Corollary \ref%
{Open_M}).

By Lemma \ref{thintetra}, $D\in \mathrm{int}\left( \overline{\mathcal{T}}
\setminus \tilde{\mathcal{F}}_{3}\right)$, and the conclusion follows.
\end{proof}

\begin{lemma}
\label{arc} \rmcite{{z-fp}} For any convex surface $S$ and any point $x$
in $S$, $F_{x}$ is contained in a minimal (by inclusion) arc $J_{x}\subset
C(x)$, possibly reduced to a point.
\end{lemma}

\begin{lemma}
\label{dim_F4} $\dim\mathcal{F}_4=4$.
\end{lemma}

\begin{proof}
Consider $T\in\mathcal{F}_4$, $x\in T$ with $\#F_x=4$, and the arc $J_x$
given by Lemma \ref{arc}. Then (see Lemma \ref{loc-max}) $C\left( x\right)$
cannot be an $X$-tree, because this would imply $\#F_x\leq3$. Neither can $%
C\left( x\right)$ be an arc, by Lemma \ref{deg}. (The later case covers
doubly covered quadrilaterals inscribed in a semi-circle centered at $x$,
hence a subfamily of dimension $2$ in $\overline{\mathcal{T}}$.) Therefore,
one of the following statements holds:

\begin{itemize}
\item $C\left( x\right)$ is a $Y$-tree. Then $F_x$, and consequently $J_x$,
contains two leaves of $C\left( x\right)$, as well as its point of degree
three and the vertex of $T$ of degree two in $C\left( x\right)$. Call such a
tetrahedron of \textsl{first type}.

\item $C\left( x\right)$ is a $H$-tree. Then $F_x$ contains the two junction
points of $C\left( x\right)$ and two leaves, separated along $J_x$ by those
junction points.
\end{itemize}

The last case is divided into two subcases. Let's consider the directions at
point $x$ of the four segments between $x$ and the leaves of its cut-locus.
The two segments to the antipodes can be either consecutive (\textsl{second
type}) or alternated with the other two (\textsl{third type}). The families
of tetrahedra of the first and of the third type are explicitly described in 
\cite{RS1}; they have respective dimensions $3$ and $4$. The family of
tetrahedra of the second type is also $4$-dimensional, the proof being
similar to the one for the third case.
\end{proof}

\begin{lemma}
\label{M4-F4} $\mathcal{M}_4 \overset{n.d.}{\supset } \mathcal{F}_4$.
\end{lemma}

\begin{proof}
By Lemmas \ref{dim_F4} and \ref{dim_T}, $\dim\mathcal{F}_{4} = 4= \dim%
\mathcal{T} -1$, and by Corollary \ref{Open_M}, $\dim\mathcal{M}_4= \dim%
\mathcal{T}$.
\end{proof}


\section{The inclusions $\mathcal{F}_{k} \supset\mathcal{F}_{k+1}$}

\label{Sect_gmax}

In this section we treat the second chain of inclusions in our theorem,
concerning the sets $\mathcal{F}_{k}$.

\begin{lemma}
\label{F2-F3} $\mathrm{int}\left( \mathcal{F}_{2}\setminus\mathcal{F}%
_{3}\right) \neq\varnothing$.
\end{lemma}

\begin{proof}
The set of tetrahedra satisfiying the hypothesis of Lemma 4.1 with a strict
inequality is open and included in $\mathcal{F}_{2}\backslash \mathcal{F}%
_{3} $.
\end{proof}

\bigskip

Notice that $\mathrm{int}\left( \mathcal{F}_{2}\setminus \mathcal{F}_{3}
\right) $ doesn't contain only \textquotedblleft thin\textquotedblright\
tetrahedra, but also the regular one.

\begin{lemma}
\label{F3} $\dim \mathcal{F}_{3} =\dim \mathcal{T}$, hence $%
\mathrm{int}\left( \mathcal{F}_{3}\setminus \mathcal{F}_{4}\right) \neq
\varnothing $.
\end{lemma}

\begin{proof}
We construct a family of tetrahedra in $\mathcal{F}_{3}$ depending on $5$
independent parameters, and thus of maximal dimension in $\mathcal{T}$ (see
Lemma \ref{dim_T}). In fact, we construct next a family of planar and simple
polygonal domains depending on $5$ independent parameters, in such a way
that they glue to tetrahedra. Moreover, the boundary of each such polygonal
domain will correspond, after gluing, to the cut locus of a distinguished
point $x$, so we may think about these domains as being the source
unfoldings of tetrahedra in a $5$-dimensional family. Our construction
directly shows that $\#F_{x}=3$ on each resulting tetrahedron $T$.

\medskip

Consider the unit circle $S(x,1)$ and the points $v_{1},w,u,u^{\prime},u^{%
\prime\prime},w^{\prime},w^{\prime\prime}\in S(x,1)$ having respectively the
following circular coordinates: $0$, $\pi/6$, $4\pi/6$, $5\pi/6$, $13\pi/12$%
, $19\pi/12$, $11\pi/6$. A straightforward computation shows that 
\begin{equation}
\left\{ 
\begin{array}{ccc}
{||v_{1}-w}{||} & {=} & {||v_{1}-w}^{\prime\prime}{||} \\ 
{||w}{-u}{||} & {=} & {||w}^{\prime}{-u^{\prime\prime}}{||}%
\end{array}
\right.  \label{*2}
\end{equation}
and 
\begin{equation*}
\begin{array}{rcl}
{\angle v_{1}wu+\angle u^{\prime\prime}w^{\prime}w^{\prime\prime}+\angle
w^{\prime}w^{\prime\prime}v_{1}} & {=} & {2\pi+\pi/12}, \\ 
{\angle wuu^{\prime}+\angle uu^{\prime}u^{\prime\prime}+\angle
u^{\prime}u^{\prime\prime}w^{\prime}} & {=} & {2\pi+\pi/12}\text{{.}}%
\end{array}%
\end{equation*}
Therefore, there exist points $v_{2},v_{3},v_{4}$ inside the polygon $%
v_{1}wuu^{\prime}u^{\prime\prime}w^{\prime}w^{\prime\prime}$ (see Figure \ref%
{F3-F4}) 
\begin{figure}[t]
\centering
\includegraphics[width=0.45\textwidth]{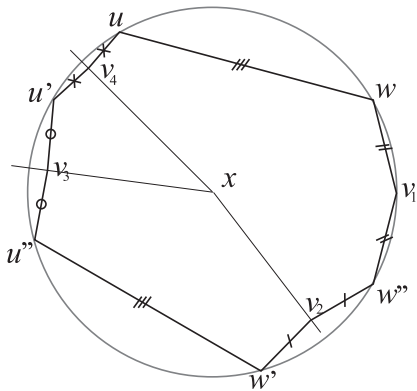}
\caption{Polygonal domain folding to a tetrahedron in $\mathcal{F}%
_{3}\setminus\mathcal{F}_{4}$.}
\label{F3-F4}
\end{figure}
such that 
\begin{equation}
\left\{ 
\begin{array}{ccc}
{||u}{-v_{4}||} & {=} & {||u^{\prime}}{-v_{4}||} \\ 
{||u^{\prime}}{-v_{3}||} & {=} & {||u^{\prime\prime}}{-v_{3}||} \\ 
{||w}^{\prime}{-v_{2}||} & {=} & {||w}^{\prime\prime}{-v_{2}||}\text{{,}}%
\end{array}
\right.  \label{*}
\end{equation}
and, moreover, 
\begin{equation}
\left\{ 
\begin{array}{ccc}
{\angle v_{1}w}{u}{+\angle u}^{\prime\prime}{w}^{\prime}{v_{2}+\angle v_{2}w}%
^{\prime\prime}{v_{1}} & {=} & {2\pi} \\ 
{\angle w}{u}{v_{4}+\angle v_{4}u}^{\prime}{v_{3}+\angle v_{3}u}%
^{\prime\prime}{w^{\prime}} & {=} & {2\pi}\text{{.}}%
\end{array}
\right.  \label{**}
\end{equation}

Let $L$ denote the closed planar set bounded by $v_{1}wuv_{4}u^{%
\prime}v_{3}u^{\prime\prime}w^{\prime}v_{2}w^{\prime\prime}$.

Glue the equal sides of $L$, according to the equalities (\ref{*}) and (\ref%
{*2}). By Alexandrov's gluing theorem (Lemma \ref{Alexandrov}), the result $%
T $ is (isometric to) a polyhedral convex surface. More precisely, the
conditions (\ref{**}) imply that $T$ is a tetrahedron, with vertices $%
v_{1},v_{2},v_{3},v_{4}$.

\medskip

Now fix $v_{1}\in S(x,1)$, and freely perturb the positions of $w$, $u$, $%
u^{\prime}$, $u^{\prime\prime}$ on $S(x,1)$, wherefrom $4$ independent
parameters. These determine, through the above conditions (\ref{*2}), the
positions of $w^{\prime}$, $w^{\prime\prime}$ on $S(x,1)$. The first
equality in (\ref{**}) determines then the position of $v_{2}$ on the
bisector of the angle $\angle w^{\prime}xw^{\prime\prime}$, but the second
inequality does not determine the position of both $v_{3}$ and $v_{4}$.
Hence we have a fifth degree of liberty.

By Alexandrov's gluing theorem, each polygonal domain obtain as above yields
a tetrahedron $T$, and Lemma \ref{Close_unf} shows that the resulting
tetrahedra are close in $\overline{\mathcal{T}}$, because their unfoldings
are close in the plane. Therefore, we have obtained a subfamily $\mathcal{L}$
of dimension $5$ in $\overline{\mathcal{T}}$. Of course, $\mathrm{dim}\left( 
\mathcal{L}\cap \mathcal{T}\right) =5$.

\medskip

We claim that the distance function $\rho_{x}^{T}$ coincides to the distance
function from $x$ in the planar domain $L$ producing $T$. Assume this is not
true, hence there exists a point $y$ in $L$ such that $\rho^T(x,y) < ||x-y||$%
. Let $\gamma$ be a segment on $T$ from $x$ to $y$, and denote by $\bar{%
\gamma}$ the image of $\gamma$ in $L$, under the unfolding of $T$ to $L$.
Because $L$ is star-shaped with respect to $x$, $\bar{\gamma} \cap \mathrm{bd%
}(L) \neq \emptyset$. Consider $z \in \bar{\gamma} \cap \mathrm{bd}(L)$, say 
$z \in uv_4$, hence the point $z^{\prime}\in u^{\prime}v_4$ given by $%
||u-z||=||u^{\prime}-z^{\prime}||$ also belongs to $\bar{\gamma} \cap 
\mathrm{bd}(L)$. Assume, for the simplicity of the presentation, that $\bar{%
\gamma} \cap \mathrm{bd}(L)=\{z,z^{\prime}\}$. (The case $\# \left( \bar{%
\gamma} \cap \mathrm{bd}(L) \right)>2$ follows by a straightforward
induction.) The triangle inequality implies 
\begin{equation*}
||x-y|| < ||x-z^{\prime}|| + ||z^{\prime}-y||= ||x-z|| + ||z^{\prime}-y||=
\rho^T(x,y),
\end{equation*}
and a contradiction is obtained, proving the claim.

The claim shows that $\mathrm{bd}\left( L\right) $ yields $C(x)$ after
gluing, so one can easily check that $\#F_{x}=3$ on $T$, and thus $\mathcal{L%
}\cap\mathcal{T}\subset\mathcal{F}_{3}$.
\end{proof}
\bigskip

Lemmas \ref{M3-F3}, \ref{F2-F3} and \ref{F3} imply the following

\begin{corollary}
$\mathcal{M}_3 \spset \mathcal{F}_3$ and $\mathcal{F}_2 \spset \mathcal{F}_3$.
\end{corollary}

\begin{lemma}
$\mathcal{F}_{4}\overset{n.d.}{\subset }\mathcal{F}_{3}$.
\end{lemma}

\begin{proof}
We first claim that any tetrahedron $T\in \mathrm{cl}\left( \mathcal{F}%
_{4}\right) $ is limit of a sequence of tetrahedra of $\mathrm{int}\left(
\mathcal{F}_{3}\right) $. Clearly it is sufficient to consider $T$ in a dense subset
of $\mathcal{F}_{4}$, so we can assume that $T$ is of second or third type
(see Lemma \ref{dim_F4} and \cite{RS1}). Let $x$ be a point on $T$ with $4$
farthest points. The cut-locus of $x$ is an $H$-tree; let $v$, $w$ be the
extremities of one of the vertical bars of the $H$. Exactly one of those
vertices, say $v$, belongs to $F_{x}$. Denote by $p$ the point of degree three
of $C\left( x\right) $ adjacent to $v$ and $w$. Consider the source
unfolding of $T$ with respect to $x$. All points of $F_x$ lie on a circle centered %
at $x$ and of radius $\rho(x,F_x)$, and $w$ is inside this circle. Now, if one radially
moves $v$ inside the circle, and radially moves $w$ in such a way that the
total angle around $p$ remains $2\pi $, he obtains a new tetrahedron such
that $\#F_{x}=3$. This tetrahedron can be deformed with $5$ degrees of
liberty: four of them come from the fact that $\mathcal{F}_{4}$ is $4$%
-dimensional \cite{RS1}, and the fifth one is the radial position of $v$; %
see, {\it e.g.}, the proof of Lemma \ref{F3}.
This proves the claim. 

Assume now that there exists $T\in \mathrm{int}_{\mathcal{F}_{3}}\left( 
\mathrm{cl}\left( \mathcal{F}_{4}\right) \right) $. In other words, there
exists an open (in $\mathcal{T}$) set $U$ containing $T$ such that $\mathcal{%
F}_{3}\cap U\subset \mathrm{cl}\left( \mathcal{F}_{4}\right) $. By the
claim, one can find a tetrahedron $T^{\prime }\in \mathrm{int} \left(\mathcal{F}_3\right) \cap U$, %
{\it i.e.\,}, there exists an open set $V$ such that $T^{\prime }\in V\subset 
\mathcal{F}_{3}$. Now $V\cap U$ is open, nonempty, and included in $\mathrm{%
cl}\left( \mathcal{F}_{4}\right) $, in contradiction with $\dim %
\mathcal{F}_{4} =4$.
\end{proof}

\bigskip

\noindent\textbf{Acknowledgement. } This work was supported by the grant
PN-II-ID-PCE-2011-3-0533 of the Romanian National Authority for Scientific
Research, CNCS-UEFISCDI.


{\small \bigskip}

{\small Jo\"{e}l Rouyer }

{\small \noindent Institute of Mathematics ``Simion Stoilow'' of the
Romanian Academy \newline
P.O. Box 1-764, Bucharest 014700, ROMANIA \newline
Joel.Rouyer@imar.ro }

{\small \bigskip}

{\small Costin V\^\i lcu }

{\small \noindent Institute of Mathematics ``Simion Stoilow'' of the
Romanian Academy \newline
P.O. Box 1-764, Bucharest 014700, ROMANIA \newline
Costin.Vilcu@imar.ro }


\begin{thebibliography}{99}
\bibitem{code1} P. K. Agarwal, B. Aronov, J. O'Rourke and C. A. Schevon, 
\textit{Star unfolding of a polytope with applications}, \textit{SIAM J.
Comput.} \textbf{26} (1997), 1689-1713

\bibitem{al} A. D. Alexandrov, \textsl{Die innere Geometrie der konvexen
Fl\"achen}, Akademie-Verlag, Berlin, 1955

\bibitem{code2} A. D. Alexandrov, \textsl{Convex Polyhedra},
Springer-Verlag, Berlin, 2005

\bibitem{unfold1} B. Aronov and J. O'Rourke, \textit{Non overlap of the star
unfolding},bibliographystyle{amsplain}
\bibliography{convex}
 Discrete Comput. Geom. \textbf{8} (1992), 219--250

\bibitem{BCS} H. Buseman, \emph{Convex surfaces}, Dover, New York, 2008,
originally published in 1958 by Interscience Publishers, Inc.

\bibitem{cfg} H. T. Croft, K. J. Falconer and R. K. Guy, \textsl{Unsolved
Problems in Geometry}, Springer-Verlag, New York, 1991

\bibitem{Demaine_O'Rourke} E. D. Demaine and J. O'Rourke, \textsl{Geometric
folding algorithms. Lincages, Origami, Polyhedra}, Cambridge University
Press, 2007

\bibitem{INV} J. Itoh, C. Nara and C. V\^\i lcu, \textit{Continuous foldings
of convex polyhedra}, to appear

\bibitem{iv} J. Itoh and C. V\^{\i}lcu, \textit{Farthest points and cut loci
on some degenerate convex surfaces}, J. Geom. \textbf{80} (2004), 106-120

\bibitem{JL} J. M. Lee, \textit{Manifold and differentiable geometry},
Graduate Studies in Mathematics 107., Providence, 2004

\bibitem{MonMien1} J Rouyer, \emph{Antipodes sur le t\'etra\`edre r\'egulier}%
, J. Geom. \textbf{77} (2003), 152-170

\bibitem{MonMien4} J. Rouyer, \emph{On antipodes on a convex polyhedron},
Adv. Geom. \textbf{5} (2005), 497-507

\bibitem{MonMien5} J. Rouyer, \emph{On antipodes on a convex polyhedron (II)}%
, Adv. Geom. \textbf{10} (2010), 403-417

\bibitem{RS1} J. Rouyer and T. Sari, \emph{As many antipodes as vertices on
convex polyhedra}, Adv. Geom. \textbf{12} (2012), 43-61

\bibitem{sa} T. Sakai, \textsl{Riemannian Geometry}, Translation of
Mathematical Monographs 149, Amer. Math. Soc. 1996

\bibitem{ss-spps-86} M. Sharir and A. Schorr, \textit{On shortest paths in
polyhedral spaces}, SIAM J. Comput. \textbf{15} (1986), 193-215

\bibitem{v2} C. V\^\i lcu, \textit{Properties of the farthest point mapping
on convex surfaces}, Rev. Roum. Math. Pures Appl. \textbf{51} (2006), 125-134

\bibitem{vz2} C. V\^\i lcu and T. Zamfirescu, \textit{Multiple farthest
points on Alexandrov surfaces}, Adv. Geom. \textbf{7} (2007), 83-100

\bibitem{Zalgaller} V. A. Zalgaller, \textit{An isoperimetric problem for
tetrahedra}, J. Math. Sci. \textbf{140} (2007), 511-527

\bibitem{z-fp} T. Zamfirescu, \textit{Farthest points on convex surfaces},
Math. Z. \textbf{226} (1997), 623-630

\bibitem{ZamTAMS} Tudor Zamfirescu, \emph{Extreme points of the distance
function on a convex surface}, Trans. Amer. Math. Soc. \textbf{350} (1998),
1395-1406.
\end{thebibliography}
\end{document}